\documentclass[11pt,reqno,oneside]{amsart}

\usepackage{amsmath,amsthm} 
\usepackage{enumerate}
\numberwithin{equation}{section}
\theoremstyle{definition}
\newtheorem{Definition}{Definition}[section]
\newtheorem{Example}[Definition]{Example}
\newtheorem{Remark}[Definition]{Remark}
\newtheorem{Problem}[Definition]{Problem}

\theoremstyle{plain}
\newtheorem{Theorem}[Definition]{Theorem}
\newtheorem{Proposition}[Definition]{Proposition}

\newtheorem{Lemma}[Definition]{Lemma}

\usepackage{geometry}

\usepackage{amssymb}        
\usepackage{mathrsfs}       
\usepackage{eucal}          
\usepackage{stmaryrd}       

\usepackage{hyperref}

\usepackage[usenames,dvipsnames]{xcolor}
\usepackage{tikz}
\usetikzlibrary{arrows, matrix}
\usepackage{tikz-cd}
\usepackage{subfig}         

\newcommand{\al}{\alpha}

\newcommand{\ga}{\gamma}
\newcommand{\Ga}{\Gamma}
\newcommand{\de}{\delta}
\newcommand{\ep}{\varepsilon}

\newcommand{\la}{\lambda}
\newcommand{\La}{\Lambda}
\newcommand{\si}{\sigma}

\newcommand{\ph}{\varphi}

\newcommand{\N}{\mathbb{N}}
\newcommand{\Z}{\mathbb{Z}}

\newcommand{\C}{\mathbb{C}}
\newcommand{\K}{\Bbbk}
\newcommand{\F}{\mathbb{F}}

\newcommand{\Fgl}{\mathfrak{gl}}
\newcommand{\Fsl}{\mathfrak{sl}}

\newcommand{\Fb}{\mathfrak{b}}
\newcommand{\Fg}{\mathfrak{g}}

\newcommand{\Fn}{\mathfrak{n}}

\DeclareMathOperator{\Aut}{Aut}

\DeclareMathOperator{\End}{End}

\DeclareMathOperator{\Frac}{Frac}
\DeclareMathOperator{\Gal}{Gal}

\DeclareMathOperator{\Supp}{Supp}

\newcommand{\iv}[2]{\llbracket #1,#2 \rrbracket}
\newcommand{\un}{\underline}
\renewcommand{\hat}{\widehat}

\title[The $q$-Difference Noether Problem for Complex Reflection Groups...]{The $q$-Difference Noether Problem for Complex Reflection Groups and Quantum OGZ Algebras}
\author{Jonas T. Hartwig}

\keywords{quantum Gelfand-Kirillov conjecture, non-commutative invariant theory, quantum group, Gelfand-Tsetlin basis, Galois ring}
\subjclass[2010]{Primary 17B37; Secondary 16K40, 16S35}

\address{Department of Mathematics, Iowa State University, Ames, IA-50011, USA}
\email{jth@iastate.edu}

\begin{document}
\maketitle
\begin{abstract}
For any complex reflection group $G=G(m,p,n)$, we prove that the $G$-invariants of the division ring of fractions of the $n$:th tensor power of the quantum plane is a quantum Weyl field and give explicit parameters for this quantum Weyl field. This shows that the $q$-Difference Noether Problem has a positive solution for such groups, generalizing previous work by Futorny and the author \cite{FutHar2014}. Moreover, the new result is simultaneously a $q$-deformation of the classical commutative case, and of the Weyl algebra case recently obtained by Eshmatov et al. \cite{EshFutOvsFer2015}.

Secondly, we introduce a new family of algebras called \emph{quantum OGZ algebras}. They are natural quantizations of the OGZ algebras introduced by Mazorchuk \cite{Mazorchuk1999} originating in the classical Gelfand-Tsetlin formulas.
Special cases of quantum OGZ algebras include the quantized enveloping algebra of $\mathfrak{gl}_n$  and quantized Heisenberg algebras.
We show that any quantum OGZ algebra can be naturally realized as a Galois ring in the sense of Futorny-Ovsienko \cite{FutOvs2010}, with symmetry group being a direct product of complex reflection groups $G(m,p,r_k)$.

Finally, using these results we prove that the quantum OGZ algebras satisfy the quantum Gelfand-Kirillov conjecture by explicitly computing their division ring of fractions.
\end{abstract}

\tableofcontents

\section{Introduction and Summary of Main Results}
In this paper,
$m$ is a positive integer,
$\F$ is a field of characteristic zero containing a primitive $m$:th root of unity,
$q\in\F\setminus\{0\}$ is an element which is not a root of unity, 
$\N$ is the set of positive integers and $\iv{a}{b}=[a,b]\cap\Z$.
By ``algebra'' we mean ``unital associative $\F$-algebra'', and $\otimes=\otimes_\F$.
The division ring of fractions of an Ore domain $A$ is denoted by $\Frac A$.

\subsection{The $q$-Difference Noether Problem}
Let $G$ be a finite group, acting linearly on an $\F$-vector space $V$. Then $G$ acts on the symmetric algebra on $V$, hence on its field of fractions $L$. If $\{x_1,x_2,\ldots,x_n\}$ is a basis for $V$, then $L=\F(x_1,x_2,\ldots,x_n)$. \emph{Noether's Problem} asks if the subfield $L^G$, consisting of all elements fixed by $G$, is a purely transcendental field extension of $\F$. In other words, is $L^G$ is isomorphic to a field of rational functions $\F(y_1,y_2,\ldots,y_N)$ for some $N$? In fact, if this is true, then necessarily $N=n$, see e.g. \cite[Prop.~3.1]{Humphreys1990}.

If $G$ is a complex reflection group and $V$ its defining reflection representation (assuming $m$ is large enough so that $\F$ contains the necessary roots of unity), the answer to Noether's Problem is affirmative, by virtue of the Chevalley-Shephard-Todd theorem \cite{SheTod1954} stating \textit{a fortiori} that the algebra of $G$-invariant polynomials, $\F[x_1,x_2,\ldots,x_n]^G$, is a polynomial algebra over $\F$. 

On the other hand, if $V$ is the direct product of $d>1$ copies of the reflection representation of a complex reflection group $G$, then the $G$-invariant polynomials on $V$ do not form a polynomial algebra. This follows from Molien's formula for the Poincar\'{e} series of the invariant subalgebra (see e.g. \cite[Ex.~1]{Smith1997}). Nevertheless, 
by a theorem of Miyata \cite[Rem.~3]{Miyata1971},
Noether's Problem has a positive solution in this case: it suffices to observe that $V$ contains a faithful $G$-submodule for which the conclusion holds (namely any one of the $d$ copies of the reflection representation).
The special case of $G=S_n$, the symmetric group, was established earlier by Mattuck  \cite{Mattuck1968}; see also \cite{Colin1996} for an explicit construction.
As an example, for $d=2$, Miyata's result gives an isomorphism
\begin{equation}\label{eq:commutative-example}
\F(x_1,x_2,\ldots,x_n; y_1,y_2,\ldots,y_n)^G\simeq 
\F(x_1,x_2,\ldots,x_n; y_1,y_2,\ldots,y_n),
\end{equation}
where $G$ is acting diagonally, e.g.
if $G$ is the symmetric group, then $\si(x_i)=x_{\si(i)}$ and $\si(y_i)=y_{\si(i)}$.

\emph{Non-commutative rational invariant theory} (see e.g. \cite[Ch.~5]{Dumas2010}) involves the study of division rings and their subrings of invariants under groups of automorphisms. The first main result of the present paper is a non-commutative analog of \eqref{eq:commutative-example}.

Let $\F_q[x,y]$ be the \emph{quantum plane}, defined as the algebra with generators $x,y$ and defining relation $yx=qxy$. The complex reflection group $G(m,p,n)$
where $p,n\in\N$ and $p|m$, acts naturally on the $n$:th tensor power of the quantum plane (see Section \ref{sec:q-DifferenceNoether} for details), hence on the associated division ring of fractions. 

\begin{Theorem} \label{thm:I} 
For any $m,p,n\in\N$, $p|m$ there is an isomorphism of $\F$-algebras
\begin{equation}\label{eq:main1}
\left(\Frac \F_q[x,y]^{\otimes n}\right)^{G(m,p,n)}\simeq \Frac \left(\F_{q^{m/p}}[x,y]\otimes \F_{q^m}[x,y]^{\otimes (n-1)}\right).
\end{equation}
\end{Theorem}

\begin{Remark}
The special case of Theorem \ref{thm:I} when $G(m,p,n)$ is a Weyl group of classical type was proved in \cite{FutHar2014}.
\end{Remark}

\begin{Remark}
The isomorphism \eqref{eq:main1} is a $q$-deformation \eqref{eq:commutative-example} for $G=G(m,p,n)$.
\end{Remark}

\begin{Remark}
Let $A_1^q(\F)$ denote the \emph{quantum Weyl algebra}, defined as the algebra with generators $x,y$ and defining relation $yx-qxy=1$. It is easy to prove that $A_1^q(\F)$ and $\F_q[x,y]$ have isomorphic division rings of fractions.
Therefore, the quantum plane may be replaced by the quantum Weyl algebra in \eqref{eq:main1}, which yields
\begin{equation}\label{eq:main1-Weyl}
\left(\Frac A_1^q(\F)^{\otimes n}\right)^{G(m,p,n)}\simeq \Frac \left(A_1^{q^{m/p}}(\F)\otimes A_1^{q^m}(\F)^{\otimes {(n-1)}}\right).
\end{equation}
Let $A_1(\F)$ denote the Weyl algebra with generators $x,y$ and defining relation $yx-xy=1$.
Then the isomorphism \eqref{eq:main1-Weyl} can be thought of as a $q$-deformation of the isomorphism
\begin{equation}
\left(\Frac A_1(\F)^{\otimes n}\right)^{G(m,p,n)}\simeq \Frac A_1(\F)^{\otimes n},
\end{equation}
which was established recently by Eshmatov et al. \cite[Thm.~1(a)]{EshFutOvsFer2015}, in fact with $G(m,p,n)$ replaced by an arbitrary finite complex reflection group. In the quantum case it is unclear how to define an action of an exceptional complex reflection group on a tensor power of the quantum plane, see Problem \ref{prb:Open-Problem1} at the end of this paper.
\end{Remark}

\subsection{Quantum OGZ Algebras}
The matrix elements in the Gelfand-Tsetlin bases for finite-dimensional irreducible $\Fgl_n$-modules
\cite{GelTse1950}\cite{Molev2007} give rise to an embedding of the enveloping algebra $U(\Fgl_n)$ into the $G$-invariants of a skew group algebra $L\ast\mathcal{M}$ where $L$ is a field of rational functions, $\mathcal{M}$ is a free abelian group acting by additive shifts on $L$, and $G$ is the direct product of symmetric groups $S_1\times S_2\times\cdots\times S_n$. The notion of a \emph{Galois ring} \cite{FutOvs2010} is an axiomatization of this situation.

In \cite{Mazorchuk1999}, Mazorchuk used a faithful generic Gelfand-Tsetlin $U(\Fgl_n)$-module (see  \cite{Mazorchuk2001a}) as the starting point for a generalization of $U(\Fgl_n)$. The resulting family of algebras are called \emph{Orthogonal Gelfand-Zetlin (OGZ) Algebras}, denoted $U(\boldsymbol{r})$, where $\boldsymbol{r}=(r_1,r_2,\ldots,r_n)\in\N^n$ is the \emph{signature}. Intuitively, $r_k$ is the number of boxes in the $k$:th row of a Gelfand-Tsetlin type tableaux. Taking $\boldsymbol{r}=(1,2,\ldots,n)$ one recovers an algebra isomorphic to $U(\Fgl_n)$.
There are also OGZ algebras of type BD, defined in \cite[Sec.~8]{Mazorchuk2001b}.

Quantum analogs of the Gelfand-Tsetlin bases were constructed in \cite{UenTakShi1990} and of Gelfand-Tsetlin modules in \cite{MazTur2000}. Thus, it is natural to ask for a quantum analog of OGZ algebras. This idea was briefly mentioned as a possibility in \cite[Rem.~3.7]{MazPonTur2003}, but no algebras were defined.

In Section \ref{sec:QuantumOGZAlgebras-Def} we propose a definition of quantum OGZ algebras (of type $A$). We denote them by $U_q^{m,p}(\boldsymbol{r})$ where $(m,p)\in\N^2$, $p|m$ and $\boldsymbol{r}$ is as before. The integers $m$ and $p$ are related to the choice of notion of $q$-number and also to the parameters of a complex reflection group. Following \cite{Mazorchuk1999}, we define these algebras as algebras of operators acting on an infinite-dimensional vector space. We motivate our definition by an informal process of quantization that we describe. A new feature is the necessity to include a certain prefactor not present in the classical case, to ensure that the formulas have the correct symmetry. We also give examples including $U_q(\Fgl_n)$ and quantized Heisenberg algebras, which show that this definition is sensible.

Our second main theorem is Theorem \ref{thm:Galois-Ring-Realization}, in which we prove that quantum OGZ algebras are examples of Galois rings. On the one hand, this gives an alternative way to define quantum OGZ algebras, which in our mind is preferable to the operator algebra approach. On the other hand it gives the possibility of studying the structure and representation theory for quantum OGZ algebras using Galois rings. The final part of the paper illustrates the latter point.

\subsection{The Quantum Gelfand-Kirillov Conjecture}
A \emph{quantum Weyl field} is the division ring of fractions of a tensor product of quantum planes parametrized by powers of $q$ (see Definition \ref{dfn:Quantum-Weyl-Field}).
An Ore domain $A_q$ depending on $q$ is said to satisfy the \emph{quantum Gelfand-Kirillov conjecture}  if $\Frac A_q$ is isomorphic to a quantum Weyl field over a purely transcendental field extension of the ground field. For $A_q=U_q(\Fn)$ where $\Fn$ is the nilpotent radical of a Borel subalgebra $\Fb$ of a semi-simple Lie algebra $\Fg$, this was first conjectured by Feigin in a talk at RIMS in 1992 and independently proved for generic $q$ by Alev-Dumas \cite{AleDum1994}, Iohara-Malikov \cite{IohMal1994} and Joseph \cite{Joseph1995}.
Since then, many other algebras in the area of quantum groups have been proved to have this property, including quantized function algebras \cite{Caldero2000},
$U_q(\Fb)$ and generalizations \cite{Panov2001}, quantum Schubert cells and CGL extensions \cite{Cauchon2003}, $U_q(\Fsl_n)$ \cite{Fauquant-Millet1999}, $U_q(\Fgl_n)$ with explicit parameters of the quantum Weyl field \cite{FutHar2014}.

Our third and final main result is an explicit description of the division ring of fractions of quantum OGZ algebras, which shows that they also satisfy the quantum Gelfand-Kirillov conjecture.
\begin{Theorem}\label{thm:III} 
Let $n\in\N$, $\boldsymbol{r}=(r_1,r_2,\ldots,r_n)\in\N^n$, $(m,p)\in\N^2$ with $p\big|m$, and
$q\in\F$ be nonzero and not a root of unity.
The quantum OGZ algebra $U_q^{m,p}(\boldsymbol{r})$ satisfies the quantum Gelfand-Kirillov conjecture. Explicitly, there is an isomorphism of $\F$-algebras
\begin{equation}
\Frac\left(U_q^{m,p}(\boldsymbol{r})\right)\simeq \K_{\bar{q}}(\un{x},\un{y})
\end{equation}
where $\bar{q}=(q^{m/p},q^{m/p},\ldots,q^{m/p},q^m,q^m,\ldots,q^m)$ with $q^{m/p}$ appearing $n-1$ times and $q^m$ appearing $r_1+r_2+\cdots+r_{n-1}-(n-1)$ times, and $\K$ is a purely transcendental field extension of $\F$ of degree $r_n$.
\end{Theorem}
\begin{Remark}
In the special case when $\boldsymbol{r}=(1,2,\ldots,n)$ and $(m,p)=(2,2)$ we recover the result for $U_q(\Fgl_n)$ from \cite{FutHar2014}, in view of Example \ref{ex:Uqgln}.
\end{Remark}

\section{The $q$-Difference Noether Problem for Complex Reflection Groups}
\label{sec:q-DifferenceNoether}

Let $\bar q=(q_1,q_2,\ldots,q_n)$ where for each $i\in\iv{1}{n}$ we have $q_i=q^{k_i}$ for some $k_i\in\Z$.
Let $\F_{\bar q}[\un{x},\un{y}]$ be the algebra generated by $\{x_i\}_{i\in\iv{1}{n}}\cup\{y_i\}_{i\in\iv{1}{n}}$ modulo the relations 
\begin{equation}
x_ix_j=x_jx_i,\quad y_iy_j=y_jy_i,\quad y_ix_j=q_i^{\delta_{ij}}x_jy_i
\end{equation}
for all $i,j\in\iv{1}{n}$. If $k_1=k_2=\cdots=k_n=1$ we write $q$ instead of $\bar q$.
It is well-known that $\F_{\bar q}[\un{x},\un{y}]$ is an Ore domain, and thus has a division ring of fractions denoted by $\F_{\bar q}(\un{x},\un{y})=\Frac(\F_{\bar q}[\un{x},\un{y}])$.
\begin{Definition}[Quantum Weyl Field]\label{dfn:Quantum-Weyl-Field}
A division ring is a \emph{quantum Weyl field over $\F$} if it is isomorphic to 
$\F_{\bar q}(\un{x},\un{y})$ for some $\bar q$.
\end{Definition}

The \emph{$q$-Difference Noether Problem} for a finite group $G$ acting by $\F$-algebra automorphisms on the division ring $\F_q(\un{x},\un{y})$ asks if the invariant subring $\F_q(\un{x},\un{y})^G$ is a quantum Weyl field.
\begin{Problem}[The $q$-Difference Noether Problem for $G$]
Do there exist $(k_1,k_2,\ldots,k_n)\in\Z^n$ and an $\F$-algebra isomorphism
\begin{equation}
\F_q(\un{x},\un{y})^G \simeq \F_{\bar q}(\un{x},\un{y}),
\end{equation}
where $\bar q=(q^{k_1},q^{k_2},\ldots,q^{k_n})$?
\end{Problem}

In this section we give a positive solution to this problem for all complex reflection groups $G(m,p,n)$, acting in a natural way on $\F_q(\un{x},\un{y})$, in the case when $q$ is not a root of unity.

Let $\boldsymbol{\mu}_m$ denote the group of all $m$:th roots of unity in $\F$.
Let $A(m,p,n)$ be the subgroup of $(\boldsymbol{\mu}_m)^n$ consisting of all $\al=(\al_1,\al_2,\ldots,\al_n)\in(\boldsymbol{\mu}_m)^n$ with $(\al_1\al_2\cdots\al_n)^{m/p}=1$.
The symmetric group $S_n$ acts naturally on $A(m,p,n)$ by permuting the entries $\al_i$.
By definition, 
\begin{equation}\label{eq:G(m,p,n)-Definition}
G(m,p,n)=S_n\ltimes A(m,p,n).
\end{equation}
The group $G(m,p,n)$ acts by $\F$-algebra automorphisms on $\F_q[\un{x},\un{y}]$, hence on the quantum Weyl field $\F_q(\un{x},\un{y})$ through
\begin{subequations}\label{eq:G-action-on-xy}
\begin{align}
g(x_i) &= \al_i x_{\si(i)} \\ 
g(y_i) &= y_{\si(i)} 
\end{align}
\end{subequations}
for all $g=(\si,\al)\in G(m,p,n)$ and all $i\in\iv{1}{n}$.

\begin{Remark}
Changing the $G(m,p,n)$-action on $y_i$ to $g(y_i) = \al_i y_{\si(i)}$ is equivalent to the above via a change of generators $x_i\mapsto x_i$, $y_i\mapsto x_iy_i$.
\end{Remark}

We now prove Theorem \ref{thm:I}, which gives a positive solution to the $q$-difference Noether problem for the reflection groups $G(m,p,n)$. The idea of the proof is to reduce to the special case $G(1,1,n)=S_n$, which was established in \cite{FutHar2014}.
\begin{proof}[Proof of Theorem \ref{thm:I}]
\un{Step 1}: There is an isomorphism
\begin{equation}\label{eq:q-Noether1}
\F_{q^m}(\un{x},\un{y})^{G(1,1,n)} \simeq \F_q(\un{x},\un{y})^{G(m,1,n)}.
\end{equation}
To see this, first note that there is an isomorphism 
\begin{equation}\label{eq:q-Noether1-proof}
 \F_{q^m}[\un{x},\un{y}] \longrightarrow \F_q[\un{x},\un{y}]^{A(m,1,n)}
\end{equation}
given by $x_i\mapsto x_i^m$ and $y_i\mapsto y_i$. Taking division ring of fractions on both sides of \eqref{eq:q-Noether1-proof}, and then taking invariants on both sides with respect to $G(1,1,n)$, we obtain \eqref{eq:q-Noether1}, using that $G(m,1,n)=G(1,1,n)\ltimes A(m,1,n)$.

\un{Step 2}: The algebra $\F_q(\un{x},\un{y})^{G(m,p,n)}$ is free of rank $p$ as a left module over $\F_q(\un{x},\un{y})^{G(m,1,n)}$, with basis $\big\{(x_1x_2\cdots x_n)^{km/p}\mid k\in\iv{0}{p-1}\big\}$. That is, we have
\begin{equation}\label{eq:q-Noether2}
\F_q(\un{x},\un{y})^{G(m,p,n)}=\bigoplus_{k=0}^{p-1} \F_q(\un{x},\un{y})^{G(m,1,n)} (x_1x_2\cdots x_n)^{km/p}.
\end{equation}
To see this, first observe that 
 $G(m,1,n)\ni(\si,\al)\mapsto \prod_i \al_i\in \boldsymbol{\mu}_m$ composed with the canonical projection $\boldsymbol{\mu}_m\to\boldsymbol{\mu}_p$ induces an isomorphism
\[G(m,1,n)/G(m,p,n)\simeq\boldsymbol{\mu}_p.\]
Let $\al=(\al_1,\al_2,\ldots,\al_n)\in (\boldsymbol{\mu}_m)^n$ be a representative for a generator of $G(m,1,n)/G(m,p,n)$.
Then,
\begin{equation} \label{eq:q-Noether-zeta}
\al\left((x_1x_2\ldots x_n)^{m/p}\right)=\ep_p(x_1x_2\cdots x_n)^{m/p},
\end{equation}
where $\ep_p=(\al_1\al_2\cdots\al_n)^{m/p} $ is a primitive $p$:th root of unity.
It is now straightforward to see that \eqref{eq:q-Noether2} is nothing but the eigenspace decomposition of $\al$ acting on $\F_q(\un{x},\un{y})^{G(m,p,n)}$. Indeed, by \eqref{eq:q-Noether-zeta}, the $k$:th term in the right hand side of \eqref{eq:q-Noether2} is clearly contained in the eigenspace of $\al$ with eigenvalue $\ep_p^k$. Since $\al$ has order $p$ there can be no other eigenvalues than $\ep_p^k$ for $k\in\iv{0}{p-1}$. Conversely, if $f$ is any element of $\F_q(\un{x},\un{y})^{G(m,p,n)}$ such that $\al(f)=\ep_p^k f$, then $f\cdot (x_1x_2\cdots x_n)^{-km/p}$ is fixed by $\al$, hence fixed by $G(m,1,n)$ and therefore belongs to $\F_q(\un{x},\un{y})^{G(m,1,n)}$. Then $f\in \F_q(\un{x},\un{y})^{G(m,1,n)}(x_1x_2\cdots x_n)^{km/p}$.

\un{Step 3}:
By \cite[Thm.~3.10]{FutHar2014} there exists an isomorphism
\begin{equation}
f:\F_q(\un{x},\un{y})\longrightarrow \F_q(\un{x},\un{y})^{G(1,1,n)}
\end{equation}
such that 
\begin{equation}
f(x_1)=x_1x_2\cdots x_n.
\end{equation}
Composing $f$, with $q$ replaced by $q^m$, with the isomorphism \eqref{eq:q-Noether1}, we obtain an isomorphism
\[g: \F_{q^m}(\un{x},\un{y}) \longrightarrow  \F_q(\un{x},\un{y})^{G(m,1,n)} 
\]
satisfying
\[g(x_1)=(x_1x_2\cdots x_n)^m.\]
Now we adjoin to $\F_{q^m}(\un{x},\un{y})$ the $p$:th root of $x_1$. The resulting division ring is isomorphic to $\F_{\bar q}(\un{x},\un{y})$ where $\bar q=(q^{m/p},q^m,q^m,\ldots,q^m)$.
 Via $g$ this corresponds to adjoining $(x_1x_2\cdots x_n)^{m/p}$ to $\F_q(\un{x},\un{y})^{G(m,1,n)}$. By \eqref{eq:q-Noether2}, the resulting algebra is isomorphic to $\F_q(\un{x},\un{y})^{G(m,p,n)}$.
\end{proof}

\section{Quantum OGZ Algebras and their Galois Ring Realization}

\subsection{Galois Rings}
Galois rings were introduced by Futorny and Ovsienko in \cite{FutOvs2010}. We briefly recall their definition and state a result we will use.

Let $\Ga$ be an integral domain, $K$ its field of fractions, $K\subseteq L$ a finite Galois extension with Galois group $G$, $\mathcal{M}$ a monoid acting by automorphisms on $L$.
We assume that the action of $\mathcal{M}$ on $L$ is \emph{$K$-separating}:
\[\text{If $h,h'\in\mathcal{M}$ satisfy $h(a)=h'(a)$ for all $a\in K$ then $h=h'$.}\]
In particular $L$ is a faithful $\mathcal{M}$-module and we may regard $\mathcal{M}\subseteq \Aut(L)$.
Assume $G$ acts by conjugation on $\mathcal{M}$, and hence naturally by automorphisms on the skew monoid ring $\mathcal{L}=L\ast\mathcal{M}$. Let $\mathcal{K}=\mathcal{L}^G$ be the subring of invariants.

\begin{Definition}[Galois $\Ga$-Ring \cite{FutOvs2010}] \label{def:GaloisRing}
A $\Ga$-subring $U\subseteq \mathcal{K}$ is a \emph{Galois $\Ga$-ring} if the following 
property holds:
\begin{equation}
UK=\mathcal{K}=KU.
\end{equation}
\end{Definition}
We will use the following result from \cite{FutOvs2010} (see also \cite[Prop.~5.2]{FutHar2014} for a more detailed proof)
which gives a sufficient criterion for a subring $U\subseteq \mathcal{K}$ to be a Galois $\Ga$-ring.
For $a=\sum_{\varphi\in\mathcal{M}} a_\varphi\varphi\in L\ast\mathcal{M}$, the \emph{support} of $a$  is defined as $\Supp(a)=\{\varphi\in\mathcal{M}\mid a_{\varphi}\neq 0\}$.
\begin{Proposition}[{\cite[Prop.~4.1(1)]{FutOvs2010}}]
\label{prp:Galois-Ring-Sufficient-Condition}
Let $U$ be a $\Ga$-subring of $\mathcal{K}$ generated as a $\Ga$-ring by $u_1,u_2,\ldots,u_n\in\mathcal{K}$. If $\cup_{i=1}^n\Supp(u_i)$ generates $\mathcal{M}$ as a monoid, then $U$ is a Galois $\Gamma$-ring in $\mathcal{K}$.
\end{Proposition}

\subsection{Definition of Quantum OGZ Algebras}\label{sec:QuantumOGZAlgebras-Def}
Let $n\in\N$, and $\boldsymbol{r}=(r_1,r_2,\ldots,r_n)\in\N^n$. Let $(m,p)\in\N^2$ with $p\big|m$.
Let $q\in\F$ be nonzero and not a root of unity.

Let $\Lambda=\F\big[x_{ki}^{\pm 1}\mid i\in\iv{1}{r_k},k\in\iv{1}{n}\big]$ be a Laurent polynomial algebra in $|\boldsymbol{r}|=r_1+r_2+\cdots+r_n$ variables, and $L$ be the field of fractions of $\Lambda$.
Let $G_{r_k}$ be the complex reflection group $G(m,p,r_k)$
(see \eqref{eq:G(m,p,n)-Definition} for the definition) and put
 $G=G_{r_1}\times G_{r_2}\times\cdots\times G_{r_n}$.
The group $G$ acts naturally on $\Lambda$, hence on $L$, by
\begin{equation}
g(x_{ki})=\al_{ki}x_{k\si_k(i)}
\end{equation}
for 
\begin{equation}\label{eq:g-element}
g=(\si_1\al_1,\si_2\al_2,\ldots,\si_n\al_n)\in G, 
\quad \si_k\in S_{r_k},
\quad \al_k=(\al_{k1},\al_{k2},\ldots,\al_{kr_k})\in A(m,p,r_k).
\end{equation}
Define $\Ga=\La^G$, the subalgebra of $G$-invariants in $\La$.
By the Chevalley-Shephard-Todd Theorem \cite{SheTod1954}\cite{Broue2010}, $\Ga$ is a semi-Laurent polynomial algebra in the generators
\begin{subequations}\label{eq:Gamma-Generators}
\begin{gather}
\label{eq:Gamma-Generators1}
\ga_{kd}=e_{kd}(x_{k1}^m,x_{k2}^m,\ldots,x_{kr_k}^m), \quad d\in\iv{1}{r_k-1},\quad k\in\iv{1}{n}, \\
\label{eq:Gamma-Generators2}
\ga_{kr_k}^{\pm 1}=(x_{k1}x_{k2}\cdots x_{kr_k})^{\pm m/p},\quad k\in\iv{1}{n}.
\end{gather}
\end{subequations}
where $e_{kd}$ denotes the elementary symmetric polynomial of degree $d$ in $r_k$ variables.

Let $\hat{\mathcal{M}}=\Z^{|\boldsymbol{r}|}$, written multiplicatively, with $\Z$-basis $\big\{\delta^{ki}\mid k\in\iv{1}{n}, i\in\iv{1}{r_k}\big\}$, and $\mathcal{M}\simeq\Z^{r_1+r_2+\cdots+ r_{n-1}}$ be the subgroup generated by $\delta^{ki}$ for $i\in\iv{1}{r_k}$,
$k\in\iv{1}{n-1}$.

Let $\mathscr{V}$ be the skew group algebra $L\ast \hat{\mathcal{M}}$, where $\delta^{ki}x_{lj}=q^{-\delta_{kl}\delta_{ij}}x_{lj}\delta^{ki}$.
In particular $\mathscr{V}$ is a left $L$-vector space with basis $\{v_\la\}$ indexed by $\hat{\mathcal{M}}$:
\begin{equation}
\mathscr{V}=\bigoplus_{\la\in\hat{\mathcal{M}}} L v_\la
\end{equation}
For $\ga\in \Ga$, let ${}^L\ga\in\End_\F(\mathscr{V})$ be the operator of left multiplication by $\ga$. For $k\in\iv{1}{n-1}$, define the operators ${}^LX_k^{\pm}\in\End_\F(\mathscr{V})$ to be the left multiplication by
\begin{equation}
X_k^{\pm} = \sum_{i=1}^{r_k} (\delta^{ki})^{\pm 1}\cdot A_{ki}^{\pm} \in L\ast \mathcal{M}\subseteq L\ast\hat{\mathcal{M}},
\end{equation}
where
\begin{equation}\label{eq:A_ki-Definition}
A_{ki}^{\pm}=\mp x_{ki}^{-\frac{m}{p}(r_{k\pm 1}-r_k)}\frac{\prod_j \left\{ \left(x_{k\pm 1,j}x_{ki}^{-1}\right)^{-m/p} \left((x_{k\pm 1,j}x_{ki}^{-1})^m-1\right)/\left(q^{-m/p}(q^m-1)\right)\right\}}{\prod_{j\neq i} \left\{\left(x_{kj}x_{ki}^{-1}\right)^{-m/p} \left((x_{kj}x_{ki}^{-1})^m - 1\right)/\left(q^{-m/p}(q^m-1)\right)\right\}}.
\end{equation}
The products are taken over all $j$ for which the expressions are defined. Thus in the numerator $j$ runs over $\iv{1}{r_{k\pm 1}}$ and in the denominator $j$ runs over $\iv{1}{r_k}\setminus\{i\}$ and  $r_0=0$ by convention.

We make the following definition, inspired by the classical case considered in \cite{Mazorchuk1999}.
\begin{Definition}[Quantum OGZ Algebra]
The \emph{$(m,p)$-form of the Quantum OGZ Algebra of signature $\boldsymbol{r}$}, denoted $U_q^{m,p}(\boldsymbol{r})$, is defined as the $\F$-subalgebra of $\End_{\F}(\mathscr{V})$ generated
by ${}^L\ga$ for $\ga\in\Ga$ and ${}^LX_k^{\pm}$ for $k\in\iv{1}{n-1}$.
\end{Definition}

Some comments are in order. The fraction in \eqref{eq:A_ki-Definition} can be informally obtained from the corresponding classical expression $\mp\frac{\prod_j (t_{k\pm 1,j}-t_{ki})}{\prod_{j\neq i}(t_{kj}-t_{ki})}$ by replacing the differences $t_{ab}-t_{cd}$ with their $q$-analog $[t_{ab}-t_{cd}]_q^{m,p}$ where we use what we call the \emph{$(m,p)$-form $q$-number}
\begin{equation}\label{eq:mpq-numbers}
[x]_q^{m,p}=\frac{q^{-xm/p}(q^{xm}-1)}{q^{-m/p}(q^m-1)}
\end{equation}
and then  replacing $q^{t_{ki}}$ by $x_{ki}$.
The prefactor $x_{ki}^{-\frac{m}{p}(r_{k\pm 1}-r_k)}$ in \eqref{eq:A_ki-Definition} has no classical counterpart (more precisely, the classical limit is $1$) but is needed in the quantum case to ensure that the elements $X_k^\pm$ are $G$-invariant relative to the natural action
\eqref{eq:G-action-on-L-ast-M} (see the proof of Theorem \ref{thm:Galois-Ring-Realization} for details). 
The reason we choose to work with $q$-numbers of the form \eqref{eq:mpq-numbers} is two-fold. On the one hand it allows us to encapsulate several different flavors of the quantum group $U_q(\mathfrak{gl}_n)$ and the quantized Heisenberg algebra (see Section \ref{sec:Examples} for examples). On the other hand, in this setup the natural ``Weyl group of type $A_{n-1}$'' turns out to be precisely the complex reflection group $G(m,p,n)$, which allows us to employ Theorem \ref{thm:I} to calculate the division ring of fractions of $U_q^{m,p}(\boldsymbol{r})$.

\subsection{Examples}\label{sec:Examples}
Before continuing, we provide some examples of quantum OGZ algebras.

\begin{Example}[Quantized $\mathfrak{gl}_n$] \label{ex:Uqgln}
Let $\boldsymbol{r}=(1,2,\ldots,n)$ and $(m,p)=(2,2)$. Then $U_q^{m,p}(\boldsymbol{r})\simeq U_q(\mathfrak{gl}_n)$, the quantized enveloping algebra of $\mathfrak{gl}_n$, defined as the algebra with generators $E_i^+, E_i^-, K_j, K_j^{-1}$ for $i\in\iv{1}{n-1}$, $j\in\iv{1}{n}$ and defining relations
\begin{subequations}
\begin{align}
K_iK_i^{-1}&=K_i^{-1}K_i=1, & K_iK_j&=K_jK_i,\\
\label{eq:q-Chevalley}
K_iE_j^\pm K_i^{-1}&=q^{\pm (\delta_{ij}-\delta_{i,j+1})} E_j^\pm, &
[E_i^+,E_j^-]&=\delta_{ij}\frac{K_i/K_{i+1}-(K_i/K_{i+1})^{-1}}{q-q^{-1}},
\end{align}
\begin{equation}
\label{eq:q-Serre}
(E_i^\pm)^2E_j^\pm - [2]_q E_i^\pm E_j^\pm E_i^\pm + E_j^\pm (E_i^\pm)^2=0, \quad |i-j|=1,
\end{equation}
\end{subequations}
where $[2]_q=(q^2-q^{-2})/(q-q^{-1})=q+q^{-1}$.
This follows from Theorem \ref{thm:Galois-Ring-Realization} below, combined with \cite[Prop.~5.9]{FutHar2014}.
Note that the correct Weyl group $W$ for $U_q(\mathfrak{gl}_n)$ (in the sense of the image of the quantum Harish-Chandra isomorphism being the $W$-invariants of the torus) is the complex reflection group $G(2,2,n)$ (the Weyl group of type $D_n$).
\end{Example}

\begin{Example}[Simply Connected Form of Quantized $\mathfrak{gl}_n$]
Let $\boldsymbol{r}=(1,2,\ldots,n)$, $m=4$, $p=2$. Then $U_q^{m,p}(\boldsymbol{r})$ is isomorphic to the so called \emph{simply connected} quantum group $\breve{U}_q(\mathfrak{gl}_n)$ (see \cite[Sec.~7.3.1]{KliSch1997}, although their $q^{1/2}$ is our $q$), obtained by changing the notion of $q$-number from $[a]_q=\frac{q^a-q^{-a}}{q-q^{-1}}$ to $[a]_q^{4,2}=[a]_{q^2}=\frac{q^{2a}-q^{-2a}}{q^2-q^{-2}}$.
Thus the right hand side of the second relation in \eqref{eq:q-Chevalley} should be replaced by
\begin{equation}
\delta_{ij}\frac{(K_i/K_{i+1})^2-(K_i/K_{i+1})^{-2}}{q^2-q^{-2}},
\end{equation}
while in \eqref{eq:q-Serre}, $[2]_q$ should be replaced by $[2]_{q^2}=q^2+q^{-2}$. The other relations stay the same.
Here the correct Weyl group is $G(4,2,n)$.
\end{Example}

\begin{Example}[The $(m,p)$-Form of Quantized $\mathfrak{gl}_n$]
Generalizing the above, we let $\boldsymbol{r}=(1,2,\ldots,n)$, $(m,p)\in\N$ with $p\big|m$. Then $U_q^{m,p}(\boldsymbol{r})$ is isomorphic to the algebra $U_q^{m,p}(\mathfrak{gl}_n)$ obtained from the definition of $U_q(\mathfrak{gl}_n)$ by using the $(m,p)$-form $q$-number \eqref{eq:mpq-numbers}.
Thus the definition of $U_q^{m,p}(\mathfrak{gl}_n)$ is the same as that of $U_q(\mathfrak{gl}_n)$ except that the right hand side of the second relation in \eqref{eq:q-Chevalley} should be replaced by
\begin{equation}
\delta_{ij}\frac{(K_i/K_{i+1})^{-m/p}\left((K_i/K_{i+1})^m-1\right)}{q^{-m/p}(q^m-1)},
\end{equation}
and in \eqref{eq:q-Serre}, $[2]_q$ should be replaced by $[2]_q^{m,p}=\frac{q^{-2m/p}(q^{2m}-1)}{q^{-m/p}(q^m-1)}=q^{-m/p}(q^m+1)$. The other relations stay the same.
The Weyl group in this case is the complex reflection group $G(m,p,n)$.
\end{Example}

\begin{Example}[Extended Quantized Heisenberg Algebra]
Let $\boldsymbol{r}=(1,1)$ and $(m,p)=(2,2)$. Then one can check that $U_q^{m,p}(\boldsymbol{r})$ is isomorphic to the Laurent polynomial algebra $\mathcal{H}_q[L,L^{-1}]$ where $\mathcal{H}_q$ is the \emph{quantized Heisenberg algebra} with generators $X,Y,K^{\pm 1}$ and relations
\begin{gather}
KK^{-1}=K^{-1}K=1,\\
YX=\frac{K-K^{-1}}{q-q^{-1}},\qquad 
XY=\frac{qK-q^{-1}K^{-1}}{q-q^{-1}},\\
KXK^{-1}=qX, \qquad 
KYK^{-1}=q^{-1}Y.
\end{gather}
This example can also be generalized to arbitrary $(m,p)\in\N^2$, $p|m$, which leads to an $(m,p)$-form of the quantized Heisenberg algebra, denoted $\mathcal{H}_q^{m,p}$. 
\end{Example}

\subsection{Galois Ring Realization of \texorpdfstring{$U_q^{m,p}(\boldsymbol{r})$}{Quantum OGZ Algebras}}

In this subsection we prove that quantum OGZ algebras can be realized as Galois $\Ga$-rings. Keeping the notation from Section \ref{sec:QuantumOGZAlgebras-Def}, 
define an action of $G$ on $L\ast \mathcal{M}$ by $\F$-algebra automorphisms as follows:
\begin{subequations}\label{eq:G-action-on-L-ast-M}
\begin{align}
\label{eq:G-action-on-x}
g(x_{ki}) &=\al_{ki}x_{k\si_k(i)}, \\ 
\label{eq:G-action-on-delta}
g(\de^{ki}) &=\de^{k\si_k(i)},
\end{align}
\end{subequations}
with $g$ is as in \eqref{eq:g-element}.
It is straightforward to check that this is indeed an action by algebra automorphisms.
Note that the action of  $G$ on  $\mathcal{M}$ is the conjugation action.

\begin{Lemma} The following two statements hold.
\begin{enumerate}[{\rm (a)}]
\item The action of $\mathcal{M}$ on $L$ is $K$-separating.
\item $L/K$ is a Galois extension, and $G=\Gal(L/K)$.
\end{enumerate}
\end{Lemma}
\begin{proof}
(a) For $h\in\mathcal{M}$ write $h=\prod (\delta^{ki})^{h_{ki}}$ where $h_{ki}\in\Z$. Let $k\in\iv{1}{n}$ and consider the element $a=\ga_{k1}=x_{k1}^m+x_{k2}^m+\cdots +x_{kr_k}^m\in K$. Then
\[h(a)=q^{-mh_{k1}}x_{k1}^m+q^{-mh_{k2}}x_{k2}^m+\cdots+q^{-mh_{kr_k}}x_{kr_k}^m.\]
Since $q$ is not a root of unity, the exponents are uniquely determined by $h$.

(b) See for example \cite[Ch.~3]{Broue2010}.
\end{proof}

\begin{Theorem}[Galois Ring Realization of $U_q^{m,p}(\boldsymbol{r})$]
\label{thm:Galois-Ring-Realization}
Let $n\in\N$, $\boldsymbol{r}=(r_1,r_2,\ldots,r_n)\in\N^n$, $(m,p)\in\N^2$ with $p\big|m$, and
$q\in\F$ be nonzero and not a root of unity. Put $U=U_q^{m,p}(\boldsymbol{r})$.
Then there exists an injective $\F$-algebra homomorphism
\begin{subequations}
\begin{equation}
\ph:U\longrightarrow (L\ast \mathcal{M})^G,
\end{equation}
given by
\begin{alignat}{2}
\ph({}^LX_k^\pm) &= X_k^{\pm},&&\qquad k\in\iv{1}{n-1},\\
\ph({}^L\ga) &=\ga,&&\qquad \ga\in\Ga.
\end{alignat}
\end{subequations}
Moreover, 
\begin{equation}\label{eq:GaloisRingProperty}
K\ph(U)=\ph(U)K =(L\ast\mathcal{M})^G.
\end{equation}
In other words, $U$ is isomorphic to a Galois $\Ga$-ring in $(L\ast\mathcal{M})^G$.
\end{Theorem}
\begin{proof}
Let $U'$ denote the subalgebra of $L\ast\mathcal{M}$ generated by $X_k^\pm$ and $\Ga$.
Let $\psi:U'\to U$ denote the homomorphism sending $X_k^\pm$ to ${}^LX_k^\pm$ and $\ga$ to ${}^L\ga$.
Thus $\psi$ is just the left regular representation of $U'$.
By definition of $U$, the map $\psi$ is surjective.
Since $U'$ is unital, $\psi$ is injective.
Thus $\ph=\psi^{-1}$ is an injective homomorphism from $U$ to $L\ast\mathcal{M}$.
We must show that the image, $U'$, of $\ph$ is contained in the invariant subalgebra $(L\ast\mathcal{M})^G$.
Since $\Ga=\La^G\subseteq L^G$, clearly $\Ga\subseteq (L\ast\mathcal{M})^G$.
Let $g=(\si_1\al_1,\si_2\al_2,\ldots,\si_n\al_n)\in G$ and $k\in\iv{1}{n-1}$. Then we have
\[g(X_k^\pm)=\sum_{j=1}^{r_k} g\left((\delta^{kj})^{\pm 1}\right) g(A_{kj}^\pm) = \sum_{j=1}^{r_k} \delta^{k\si_k(j)} g(A_{kj}^\pm).\]
Thus we must show that 
\begin{equation}\label{eq:A_ki-Covariance}
g(A_{ki}^\pm)=A_{k\si_k(i)}^\pm \quad\text{for all $i\in\iv{1}{r_k}$.}
\end{equation}
First rewrite $A_{ki}^\pm$ as follows:
\begin{equation}
A_{ki}^\pm = \mp \frac{\prod_j x_{k\pm 1,j}^{-m/p} \cdot\prod_j \left\{(x_{k\pm 1,j}^mx_{ki}^{-m}-1)/\left(q^{-m/p}(q^m-1)\right)\right\}}{\prod_j x_{kj}^{-m/p}\cdot \prod_{j\neq i} \left\{(x_{kj}^mx_{ki}^{-m}-1)/\left(q^{-m/p}(q^m-1)\right)\right\} }.
\end{equation}
By direct inspection, the products involving curly brackets are mapped under $g$ to the same expressions with $i$ replaced by $\si_k(i)$. The other two products belong to $\Ga$ by \eqref{eq:Gamma-Generators2}, and hence are $G$-invariant. This proves \eqref{eq:A_ki-Covariance}.
Finally, the identities \eqref{eq:GaloisRingProperty} follow from Proposition \ref{prp:Galois-Ring-Sufficient-Condition} and that the union of the support of the generators $X_k^\pm$ is equal to $\{\delta^{ki},(\delta^{ki})^{-1}\mid i\in\iv{1}{r_k}, k\in\iv{1}{n-1}\}$ which generates $\mathcal{M}$ as a monoid.
\end{proof}

\section{The Quantum Gelfand-Kirillov Conjecture for \texorpdfstring{$U_q^{m,p}(\boldsymbol{r})$}{Quantum OGZ Algebras}}
We will apply the Galois ring realization of quantum OGZ algebras (Theorem \ref{thm:Galois-Ring-Realization}) and the positive solution to the $q$-difference Noether problem (Theorem \ref{thm:I}) to give a proof of Theorem \ref{thm:III} describing the division ring of fractions of $U_q^{m,p}(\boldsymbol{r})$.

Let $\mathcal{M}_k$ be the subgroup of $\mathcal{M}$ generated by $\delta^{ki}$, $i\in\iv{1}{r_k}$, and $\La_k=\F\left[x_{ki}^{\pm 1}\mid i\in\iv{1}{r_k}\right]$.
Let $P_k=\F_q[x_1^{\pm 1},x_2^{\pm 1},\ldots,x_{r_k}^{\pm 1};y_1^{\pm 1},y_2^{\pm 1},\ldots,y_{r_k}^{\pm 1}]$ be the \emph{quantum Laurent polynomial algebra}, generated by $x_i^{\pm 1}, y_i^{\pm 1}$ for $i\in\iv{1}{r_k}$ with defining relations
\begin{subequations}\label{eq:Quantum-Laurent-Relations}
\begin{gather}
x_ix_i^{-1}=x_i^{-1}x_i=y_iy_i^{-1}=y_i^{-1}y_i=1,\\
x_ix_j=x_jx_i,\qquad y_iy_j=y_jy_i, \qquad y_ix_j=q^{\delta_{ij}}x_jy_i.
\end{gather}
\end{subequations}
Recall that $G_k=G(m,p,r_k)$. The group $G_k$ acts by automorphisms on $P_k$ as in \eqref{eq:G-action-on-xy}.

\begin{Lemma} There is a $G_k$-equivariant $\F$-algebra isomorphism
\begin{equation}\label{eq:Laurent-Isomorphism}
\psi: \F_q[x_1^{\pm 1},x_2^{\pm 1},\ldots,x_{r_k}^{\pm 1};y_1^{\pm 1},y_2^{\pm 1},\ldots,y_{r_k}^{\pm 1}]
\overset{\sim}{\longrightarrow}  \Lambda_k\ast \mathcal{M}_k 
\end{equation}
given by
\begin{equation}
\psi(y_i)=(\delta^{ki})^{-1},\qquad 
\psi(x_i)=x_{ki}.
\end{equation}
\end{Lemma}
\begin{proof} That $\psi$ is an isomorphism is immediate, using \eqref{eq:Quantum-Laurent-Relations}, that $x_{ki}$ and $\delta^{ki}$ are invertible in $\La_k\ast\mathcal{M}_k$ and satisfy
\[x_{ki}x_{kj}=x_{kj}x_{ki},\qquad \delta^{ki}\delta^{kj}=\delta^{kj}\delta^{ki},
\qquad \delta^{ki}x_{kj} = q^{-\delta_{ij}}x_{kj}\delta^{ki}.\]
Moreover $\psi$ clearly intertwines the $G_k$-actions given in \eqref{eq:G-action-on-L-ast-M} and \eqref{eq:G-action-on-xy}. 
\end{proof}

We are now ready to prove the third and final main theorem.

\begin{proof}[Proof of Theorem \ref{thm:III}]
The proof is similar to \cite[Sec.~6]{FutHar2014} but we provide some more details for clarity.
By \eqref{eq:GaloisRingProperty}, the map $\varphi$ induces an isomorphism
\[\Frac\left(U_q^{m,p}(\boldsymbol{r})\right)  \simeq
\Frac\left((L\ast\mathcal{M})^G\right).  \]
Since $G$ is finite, any fraction is equivalent to a fraction with $G$-invariant denominator, hence
\[\Frac\left((L\ast\mathcal{M})^G\right) \simeq 
\left(\Frac(L\ast\mathcal{M})\right)^G. \]
Since $L=\Frac\La$,
\begin{equation}\label{eq:qGK-proof-step0}
\left(\Frac(L\ast\mathcal{M})\right)^G \simeq 
\left(\Frac(\La\ast\mathcal{M})\right)^G. 
\end{equation}
Since $\delta^{ki}$ fixes $x_{lj}$ if $l\neq k$ and $\mathcal{M}$ is generated by $\delta^{ki}$ for $k\in\iv{1}{n-1}$, $i\in\iv{1}{r_k}$, the right hand side of \eqref{eq:qGK-proof-step0} is isomorphic to
\[
\left(\Frac\big( (\La_1\ast \mathcal{M}_1)\otimes(\La_2\ast\mathcal{M}_2)\otimes\cdots\otimes 
(\La_{n-1}\ast\mathcal{M}_{n-1})\otimes\La_n\big)\right)^G.
\] 
Using again that $G$ is finite, and is a direct product of groups $G_k$ acting trivially on $\delta^{li}$ and $x_{li}$ if $l\neq k$, the last division ring is isomorphic to
\begin{equation}\label{eq:qGK-proof-step}
\Frac\left( (\La_1\ast \mathcal{M}_1)^{G_1}\otimes(\La_2\ast\mathcal{M}_2)^{G_2}\otimes\cdots\otimes 
(\La_{n-1}\ast\mathcal{M}_{n-1})^{G_{n-1}}\otimes(\La_n)^{G_n}\right). 
\end{equation}
By the isomorphism \eqref{eq:Laurent-Isomorphism},
 $(\La_k\ast\mathcal{M}_k)^{G_k}$ is isomorphic to  $\F_q[x_1^{\pm 1},x_2^{\pm 1},\ldots,x_{r_k}^{\pm 1}; y_1^{\pm 1},y_2^{\pm 1},\ldots,y_{r_k}^{\pm 1}]^{G_k}$, which in turn has the same division ring of fractions as $\F_q[\un{x},\un{y}]^{G_k}$.
 (Here the number of pairs of variables $x_i,y_i$ is suppressed from the notation, but can be deduced from the notation $G_k$ to be $r_k$.) Hence \eqref{eq:qGK-proof-step} is isomorphic to
\[\Frac\left( \F_q[\un{x},\un{y}]^{G_1}\otimes \F_q[\un{x},\un{y}]^{G_2}\otimes\cdots\otimes 
\F_q[\un{x},\un{y}]^{G_{n-1}}\otimes(\La_n)^{G_n}\right).\]
By Theorem \ref{thm:I}, the latter is isomorphic to
\begin{equation}\label{eq:qGK-proof-step2}
\Frac\left( \F_{\bar q_1}[\un{x},\un{y}] \otimes \F_{\bar q_2}[\un{x},\un{y}]\otimes\cdots\otimes 
\F_{\bar q_{n-1}}[\un{x},\un{y}]\otimes\Ga_n\right),
\end{equation}
where $\bar{q}_k=(q^{m/p},q^m,q^m,\ldots,q^m)\in\F^{r_k}$ for $k\in\iv{1}{n-1}$,
and we put $\Ga_n=(\La_n)^{G_n}$. Thus, letting $\K=\Frac \Ga_n$, \eqref{eq:qGK-proof-step2} is isomorphic to 
$\K_{\bar q}(\un{x},\un{y})$
where $\bar q$ consists of $n-1$ components of $q^{m/p}$ and $|\boldsymbol{r}|-(r_n+n-1)$ components of $q^m$.
By the Chevalley-Shephard-Todd theorem, $\K$ is a purely transcendental field extension of $\F$ of degree $r_n$.
\end{proof}

\section{Open Problem}

We end by stating an open problem.

\begin{Problem}\label{prb:Open-Problem1}
If $G$ is an exceptional complex reflection group, does it act naturally by algebra automorphisms on a suitable tensor power of the quantum plane? If so, is there a positive solution to the $q$-Difference Noether Problem for such $G$?
\end{Problem}

\bibliographystyle{siam}

\end{document}